\documentclass[12pt]{amsart} 
\usepackage{amssymb,amsmath,amsthm,amsfonts}
\usepackage {latexsym,enumerate}
\usepackage{bbm}
\usepackage{color,mathabx}

\include{srctex}

\usepackage{marginnote}
\usepackage[letterpaper, top=1.2in, bottom=1.2in, outer=1.1in, inner=1.1in, heightrounded, marginparwidth=2.5cm, marginparsep=2mm]{geometry}

\allowdisplaybreaks
\newcommand{\nc}{\newcommand}
\nc{\les}{\lesssim}
\nc{\nit}{\noindent}
\nc{\nn}{\nonumber}
\nc{\D}{\partial}
\nc{\diff}[2]{\frac{d #1}{d #2}}
\nc{\diffn}[3]{\frac{d^{#3} #1}{d {#2}^{#3}}}
\nc{\pdiff}[2]{\frac{\partial #1}{\partial #2}}
\nc{\pdiffn}[3]{\frac{\partial^{#3} #1}{\partial{#2}^{#3}}}
\nc{\abs}[1] {\lvert #1 \rvert}
\nc{\cAc}{{\cal A}_c}
\nc{\cE}{{\cal E}}
\nc{\cF}{{\cal F}}
\nc{\cP}{{\cal P}}
\nc{\cV}{{\cal V}}
\nc{\cQ}{{\cal Q}}
\nc{\cGin}{{\cal G}_{\rm in}}
\nc{\cGout}{{\cal G}_{\rm out}}
\nc{\cO}{{\cal O}}
\nc{\Lav}{{\cal L}_{\rm av}}
\nc{\cL}{{\cal L}}
\nc{\cB}{{\cal B}}
\nc{\cZ}{{\cal Z}}
\nc{\cR}{{\cal R}}
\nc{\cT}{{\cal T}}
\nc{\cY}{{\cal Y}}
\nc{\cX}{{\cal X}}
\nc{\cXT}{{{\cal X}(T)}}
\nc{\cBT}{{{\cal B}(T)}}
\nc{\vD}{{\vec \mathcal{D}}}
\nc{\efield}{\mathcal{E}}
\nc{\vE}{{\vec \efield}}
\nc{\vB}{{\vec \mathcal{B}}}
\nc{\vH}{{\vec \mathcal{H}}}
\nc{\mR}{\mathcal{R}}
\nc{\mG}{\mathcal{G}}
\nc{\ty}{{\tilde y}}
\nc{\tu}{{\tilde u}}
\nc{\tV}{{\tilde V}}
\nc{\Pc}{{\bf P_c}}
\nc{\bx}{{\bf x}}
\nc{\bX}{{\bf X}}
\nc{\bXYZ}{{\bf XYZ}}
\nc{\oo}{\"o}
\nc{\bY}{{\bf Y}}
\nc{\bF}{{\bf F}}
\nc{\bS}{{\bf S}}
\nc{\dV}{{\delta V}}
\nc{\dE}{{\delta E}}
\nc{\TT}{{\Theta}}
\nc{\dPsi}{{\delta\Psi}}
\nc{\order}{{\cal O}}
\nc{\Rout}{R_{\rm out}}
\nc{\eplus}{e_+}
\nc{\eminus}{e_-}
\nc{\epm}{e_\pm}
\nc{\eps}{\varepsilon}
\nc{\vnabla}{{\vec\nabla}}
\nc{\G}{\Gamma}
\nc{\w}{\omega}
\nc{\mh}{h}
\nc{\mg}{g}
\nc{\vphi}{\varphi}
\nc{\tlambda}{\tilde\lambda}
\nc{\be}{\begin{equation}}
\nc{\ee}{\end{equation}}
\nc{\ba}{\begin{eqnarray}}
\nc{\ea}{\end{eqnarray}}

\nc{\g}{\gamma}
\nc{\ol}{\overline}

\newtheorem{theorem}{Theorem}[section]
\newtheorem{lemma}[theorem]{Lemma}
\newtheorem{prop}[theorem]{Proposition}
\newtheorem{corollary}[theorem]{Corollary}

\newtheorem{rmk}[theorem]{Remark}

\nc{\pT}{\partial_T}
\nc{\pz}{\partial_z}
\nc{\pt}{\partial_t}
\nc{\la}{\langle}
\nc{\ra}{\rangle}
\nc{\infint}{\int_{-\infty}^{\infty}}
\nc{\halfwidth}{6.5cm}
\nc{\figwidth}{10cm}
\newcommand{\f}{\frac}

\nc{\nlayers}{L} \nc{\nsectors}{M}
\nc{\indicator}{\mathbf{1}}
\nc{\Rhole}{R_{\rm hole}}
\nc{\Rring}{R_{\rm ring}}
\nc{\neff}{n_{\rm eff}}
\nc{\Frem}{F_{\rm rem}}
\nc{\R}{\mathbb R}
\nc{\C}{\mathbb C}
\nc{\Z}{\mathbb Z}
\nc{\N}{\mathbb N}
\nc{\DD}{\Delta}
\nc{\cD}{\mathcal D}
\nc{\Kato}{\mathcal K}
\nc{\lnorm}{\left\|}
\nc{\rnorm}{\right\|}
\nc{\rnormp}{\right\|_{\ell^{p,\eps}}}
\nc{\rar}{\rightarrow}
\begin{document}

\begin{abstract}

	We prove Strichartz estimates for the Schr\"odinger equation in $\R^n$, $n\geq 3$, with a Hamiltonian $H = -\Delta + \mu$. The perturbation $\mu$ is a compactly supported measure in $\R^n$ with dimension $\alpha > n-(1+\frac{1}{n-1})$. The main intermediate step is a local decay estimate in $L^2(\mu)$ for both the free and perturbed Schr\"odinger evolution.
\end{abstract}

\title[Strichartz Estimates with a Measure-Valued Potential]{Strichartz Estimates for the Schr\"odinger Equation with a Measure-Valued Potential}

\author{M. Burak Erdo\smash{\u{g}}an, Michael Goldberg, William R. Green}
\thanks{The first author was partially supported by NSF grant DMS-1501041.
The second author is supported by Simons Foundation Grant 281057. The third author is supported by Simons Foundation Grant 511825.}  
\address{Department of Mathematics \\
University of Illinois \\
Urbana, IL 61801, U.S.A.}
\email{berdogan@illinois.edu}

\address{Department of Mathematics\\
University of Cincinnati \\
Cincinnati, OH 45221 U.S.A.}
\email{goldbeml@ucmail.uc.edu}
\address{Department of Mathematics\\
Rose-Hulman Institute of Technology \\
Terre Haute, IN 47803, U.S.A.}
\email{green@rose-hulman.edu} 

\maketitle

%\comment{In that proof I claim that $-\Delta + \mu$ is $\lambda \in [0,\infty)$ has a finite number of negative eigenvalues.  Does this need further explanation?}
\section{Introduction}

The dispersive properties of the free Schr\"odinger semigroup $e^{it\Delta}$
are described in many ways,
with one of the most versatile estimates being the family of Strichartz
inequalities
\begin{equation} \label{eq:freeStrichartz}
\| e^{it\Delta}u\|_{L^p_t L^q_x} \les \|u\|_{L^2(\R^n)}
\end{equation}
over the range $2 \leq p,q \leq \infty$, 
$\frac{2}{p} + \frac{n}{q} = \frac{n}{2}$, except for the endpoint 
$(p,q) = (2,\infty)$.
There is a substantial body of literature devoted to establishing Strichartz
inequalities and other dispersive bounds for the linear Schr\"odinger evolution
of perturbed operators $H = -\Delta + V(x)$. \cite{Strichartz77,GiVe85,KT}  prove Strichartz inequalities for the free evolution. The latter two of these, as well as \cite{RoSc04} create a framework for extending them to perturbed Hamiltonians so long as the Schr\"odinger semigroup has suitable $L^1\to L^\infty$ dispersive bounds or $L^2(\R^n\times \R)$ smoothing.  This strategy has been used to establish Strichartz estimates for the Schr\"odinger evolution for electric \cite{JoSoSo91}, magnetic \cite{ErGoSc} and time-periodic \cite{Go09} perturbations. 
Most commonly $V(x)$ is assumed
to exhibit pointwise polynomial decay or satisfying an integrability criterion such as belonging to a space $L^r_{loc}(\R^n)$
for some $r \geq \frac{n}{2}$.
Our goal in this paper is to show that Strichartz inequalities hold for a class
of short-range potentials $V(x)$ that include measures $\mu(dx)$ as admissible local
singularities.

Measure-valued potentials are often considered in one dimension; the operator
$-\frac{d^2}{dx^2} + c\delta_0$ is often the subject of exercises in an introductory
quantum mechanics course.  In higher dimensions there are several plausible
generalizations of this example.  The three-dimensional Schr\"odinger operator
$H = -\Delta + \sum c_j \delta(x_j)$ is studied in~\cite{DaPiTe06} 
and~\cite{DeMiScYa18}, and in two dimensions in \cite{CorMiYa}.  Here, the
singularity of the potential imposes boundary conditions at each point $x_j$
for functions belonging to the domain of $H$.  As an eventual consequence,
linear dispersive and Strichartz inequalities hold only on a subset of the
range described above.

The potentials considered in this paper are less singular than a delta-function in $\R^n$,
but still not absolutely continuous with respect to Lebesgue measure.
The surface measure of a compact hypersurface $\Sigma \subset \R^n$
is a canonical example of an admissible potential we consider. 
More generally we work with compactly supported
fractal measures (on $\R^n$) of a sufficiently high dimension.
The exact threshold will be determined in context.  
Arguments regarding the self-adjointness of $H$ require a dimension
greater than $n-2$ so that multiplication by $\mu$ remains compact relative to
the Laplacian.
We are forced to increase the threshold dimension to $n - (1+ \frac{1}{n-1})$
in the proof of the local decay and Strichartz estimates.
Under these conditions, and a modest assumption about the spectral properties
of $H$, we prove that the entire family of Strichartz
inequalities~\eqref{eq:freeStrichartz} is preserved with the possible exception of the
$(p,q)=(2, \frac{2n}{n-2})$ endpoint.

With $B(x,r)$ a ball of radius $r$ centered at $x\in \R^n$,
we say that a compactly supported signed measure $\mu$ is
{\em $\alpha$-dimensional} if it satisfies
\begin{equation} \label{eq:dimension}
|\mu|(B(x,r)) \leq C_\mu r^{\alpha} \ \ {\rm for\ all}\ r>0\ {\rm and}\ 
x \in \R^n
\end{equation}
Nontrivial $\alpha$-dimensional measures exist for any $\alpha \in [0, n]$.

The first obstruction to Strichartz estimates with
a Hamiltonian $H = -\Delta + V$ is the possible existence of bound states,
functions $\psi \in L^2(\R^n)$ that solve $ H\psi = E \psi$ for some
real number $E$.  Each bound state gives rise to a solution of the 
Schr\"odinger equation $e^{itH}\psi(x) =e^{itE}\psi(x)$, which
satisfies~\eqref{eq:freeStrichartz} only for $(p,q) = (\infty, 2)$ and
no other choice of exponents.

Our main result asserts that the perturbed evolution $e^{itH}$ satisfies
Strichartz estimates once all bound states of $H$ are projected away.
We impose additional spectral assumptions that all eigenvalues of $H$ 
are strictly negative, and that there is no resonance at zero.
In this paper we say a resonance occurs at $\lambda$ when the equation
\begin{equation*}
\psi + (-\Delta - (\lambda \pm i0))^{-1}\mu\psi = 0
\end{equation*}
has nontrivial solutions belonging to the Sobolev space $\dot{H}^1(\R^n)$
but not to $L^2$ itself.  

Thanks to the compact support of $\mu$, one can easily show that resonances are
impossible at $\lambda < 0$, and can only occur at $\lambda = 0$ in dimensions
3 and 4.  We show in Section~\ref{sec:LowEnergy}
that resonances do not occur when $\lambda > 0$.  Eigenvalues at zero are
possible provided the negative part of the potential is large enough.
Positive eigenvalues are known to be absent for a wide class of potentials
(see~\cite{KoTa06}) covering some (but not all) of the measures considered
here, see Remark~\ref{rmk:eval}.

To state our results, we define the following $L^p$ spaces.  For $\mu$ a signed measure on $\R^n$, we define
\be\label{eq:LpVdef}
	L^p(\mu):=\left\{ f:\R^n\to \mathbb C \,:\, \int_{\R^n } |f|^p\, d|\mu|<\infty \right\}
\ee 
for $1\leq p< \infty$. With the natural, minor modification one can define $L^\infty(\mu)$.
It is worthwhile to note that multiplication by $\mu$ is an isometry from $L^p(\mu)$ to $L^{p'}(\mu)^*$ with $p'$ the H\"older conjugate of $p$ for any $1\leq p\leq\infty$.  This can be seen easily by using the natural duality pairing. Throughout the paper we will take particular advantage of the fact that multiplication by $\mu$ maps $L^2(\mu)$ to its dual space.  Finally, let $P_{ac}$ denotes projection
onto the continuous spectrum of $-\Delta + \mu$.

\begin{theorem} \label{thm:main}
Let $\mu$ be a compactly supported signed measure on $\R^n$ of dimension
$\alpha > n - (1 + \frac{1}{n-1})$.
If the Schr\"odinger operator $-\Delta + \mu$ has no resonance
at zero and no eigenvalues at any $\lambda \geq 0$, then for each
$f \in L^2(\R^n)$ we have the local decay bounds
\begin{align} \label{eq:freesmoothing}
\|e^{it\Delta}f\|_{L^2_t L^2(\mu)} &\les \|f\|_2 \\
\|e^{it(-\Delta + \mu)}P_{ac}f\|_{L^2_t L^2(\mu)} &\les \|f\|_2 \label{eq:smoothing}
\end{align}
and the Strichartz inequalities
\begin{equation} \label{eq:Strichartz}
\|e^{it(-\Delta+\mu)}P_{ac}f\|_{L^p_tL^q_x} \les \|f\|_{2}
\end{equation}
for admissible pairs $(p,q)$ with $\frac{2}{p} + \frac{n}{q} = \frac{n}{2}$ 
and $p > 2$. 
\end{theorem}

The second author considered $L^1 \to L^\infty$ dispersive estimates in $\R^3$ (under the same set of assumptions when $n=3$, including $\alpha > \frac32$) in~\cite{Go12}.  Strichartz inequalities in this case follow as a direct consequence by~\cite{KT}.  The results presented here in $\R^n$, $n \geq 4$, are new and rely in part on recent advances in Fourier restriction problems such as \cite{LuRo19, DuZh19}.  In particular the improved decay of spherical Fourier means allows us to capture the physically relevant $\alpha = n-1$ case where the potential might be supported on a compact hypersurface in higher dimensions.  It is known for dimensions $n>3$ that $L^1\to L^\infty$ dispersive estimates need not hold even for compactly supported potentials $V$ if $V$ is not sufficiently differentiable, see \cite{GV,ErGr10}.  Whereas the smoothness is not required for Strichartz estimates to hold in higher dimensions.  Our argument works in dimensions $n\geq 3$, technical issues in dimension $n=2$ (for example with the use of $\dot{H}^1(\R^n))$ would require different methods.
 
$\alpha$-dimensional measures also satisfy a strong Kato-type property that for any $\gamma < \alpha$,
\begin{equation} \label{eq:KatoBound}
\sup_{y\in\R^n} \int_{|x-y| < r} \frac{|\mu|(dx)}{|x-y|^{\gamma}} \les 
C_{\mu} r^{\alpha-\gamma}.
\end{equation}
Furthermore, since $\mu$ has compact support, the integral over the entire 
space $x \in \R^n$ is bounded uniformly in $y$.  These integral bounds
will be proved as Lemma~\ref{prop:dimensionKato} below.  The choice
$\gamma = n-2$ is significant due to its connection with the Green's function 
of the Laplacian in $\R^n$ when $n\geq3$.

We also characterize potentials in terms of the {\em global Kato norm}, defined
on signed measures in $\R^n$ by the quantity
\begin{equation} \label{eq:KatoNorm}
\|\mu\|_{\Kato} = \sup_{y\in\R^n} \int_{\R^n} \frac{|\mu|(dx)}{|x-y|^{n-2}}
\end{equation}
One can see that every element with finite global Kato norm is a $(n-2)$-dimensional measure with
$C_\mu \leq \|\mu\|_{\Kato}$, by comparing $|x-y|^{2-n}$ to the characteristic
function of a ball.  The converse is false, however the Kato class contains
all compactly supported measures of dimension $\alpha > n-2$.  We examine
this relationship   in Lemma~\ref{prop:dimensionKato}.   We follow the naming convention in
Rodnianski-Schlag~\cite{RoSc04} where the global Kato norm is applied to
dispersive estimates in $\R^3$, as opposed to the local norms considered in Schechter~\cite{Sc71}.

There is a now well-known strategy to obtain the Strichartz estimates \eqref{eq:Strichartz}. One uses the space-time $L^2$ estimates \eqref{eq:freesmoothing} and \eqref{eq:smoothing} and the argument of Rodnianski-Schlag, \cite{RoSc04}.  There is a minor modification to the Rodnianski-Schlag framework in that instead of factorizing the operator corresponding to multiplication by $\mu$, we instead apply it directly as a bounded map from $L^2_tL^2(\mu)$ to its dual space.

The resolvent operators $(-\Delta-\lambda)^{-1}$ and $(-\Delta+\mu-\lambda)^{-1}$ are well defined for $\lambda$ in the resolvent sets.
We define the limiting resolvent operators
\begin{equation}
R_0^\pm(\lambda) := \lim\limits_{\eps\to 0^+} (-\Delta -(\lambda \pm i \eps))^{-1}
\quad \text{and} \quad
R_\mu^\pm(\lambda) := \lim\limits_{\eps\to 0^+} (-\Delta + \mu - (\lambda \pm i\eps))^{-1}
\end{equation}
Following Kato's derivation~\cite{Ka65}, $L^2$ estimates such as \eqref{eq:freesmoothing} and \eqref{eq:smoothing} 
are valid precisely if there are uniform bounds of the resolvent operators. 
We prove the following mapping bounds for the resolvents $R_0^\pm(\lambda)$ and $R_\mu^\pm(\lambda)$.
\begin{theorem} \label{thm:resolvents}
Under the hypotheses of Theorem~\ref{thm:main}
\begin{align}
\sup_{\lambda \geq 0} \|((R_0^+(\lambda) - R_0^-(\lambda))\mu\|_{L^2(\mu)\to L^2(\mu)} &< \infty 
\label{eq:UnifFreeRes}\\
\text{and} \quad \sup_{\lambda \geq 0} \|((R_\mu^+(\lambda) - R_\mu^-(\lambda))\mu\|_{L^2(\mu)\to L^2(\mu)} &< \infty. 
\label{eq:UnifPerturbedRes}
\end{align}
\end{theorem}
Due to the different challenges of establishing these bounds when the spectral parameter is close to $\lambda=0$ (small energy) or bounded away from zero (large energy), we require different tools in each regime.  We bound the low energy contribution in Section~\ref{sec:LowEnergy} in Lemma~\ref{lem:lowenergybd}, while the large energy is controlled in Section~\ref{sec:HighEnergy}.  Once the resolvent bounds are established at all energies, we assemble the results to prove Theorem~\ref{thm:main}.

\section{Self-Adjointness and Compactness} \label{sec:selfadjoint}

For any perturbation $V(x)$ which is not a bounded function of $x$ there are well known
difficulties identifying the domain of $-\Delta + V$ and its adjoint operator.   The main goal of this section is to prove Proposition~\ref{prop:selfadj} below.
Along the way, we will prove some compactness results that will be
useful for describing the spectral measure of $-\Delta + \mu$.  

\begin{prop}\label{prop:selfadj}
	
	If $\mu$ is a compactly supported $\alpha$-dimensional signed measure for some $\alpha>n-2$, then there exists a unique self-adjoint extension of $-\Delta+\mu$. 
	
\end{prop}

The first step is to check that $\mu$ satisfies both a local and global
``Kato condition."
\begin{lemma} \label{prop:dimensionKato}
Suppose $\mu$ is an $\alpha$-dimensional signed measure with support
in the ball $B(0, M)$, and $\gamma$ is such that $\alpha > \gamma > 0$.  Then $\mu$ satisfies the estimates
\begin{equation} \label{eq:dimensionKato}
\begin{aligned}
	\sup_{y\in\R^n}\int_{|x-y|<r} \frac{|\mu|(dx)}{|x-y|^{\gamma}}
	&\les C_\mu  r^{\alpha-\gamma} \quad \textrm{\it for all}\ r>0 \\
	\textrm{and}\quad  \sup_{y \in \R^n}\int_{\R^n} \frac{|\mu|(dx)}{|x-y|^{\gamma}} &\les C_\mu M^{\alpha - \gamma}.
\end{aligned}
\end{equation}
%Consequently, if $\mu$ can be approximated in $\Kato$ by a sequence of measures $\mu_j$ with dimension $\alpha_j > 1$, then $\mu$ satisfies~\eqref{eq:localKato}. 
\end{lemma}

\begin{proof}
For each point $y \in \R^n$, 
\begin{equation*}
\begin{aligned} \int_{|x-y| < r} \frac{|\mu|(dx)}{|x-y|^\gamma}
&= \gamma \int_0^r \frac{|\mu|(B(y, t))}{t^{\gamma+1}} \,dt + r^{-\gamma} |\mu|(B(y,r)) \\
&\les C_\mu \Big( \int_0^r t^{\alpha - \gamma - 1} \, dt + r^{\alpha -\gamma}\Big) =
C_\mu\Big(1 + \frac{1}{\alpha-\gamma}\Big)r^{\alpha-\gamma}.
\end{aligned}
\end{equation*}
This establishes the first claim.  For the second claim, if $|y| < 2M$,  the global bound is achieved by setting $r = 3M$. 
If $|y| > 2M$ then the integral in~\eqref{eq:dimensionKato} is easily
bounded by $|y|^{-\gamma}|\mu|(B(0, M))$ by observing that $|x-y| \sim |y| \gtrsim M$
within the  support of $\mu$.

%Convergence of  $\mu_j$ in the global Kato norm forces the collection of
%functions $\eta_j(r) = \sup_y \int_{|x-y|< r} |x-y|^{-1}\, d|\mu_j|$
%to converge uniformly in $r$.
%The property $\lim_{r\to 0} \eta_j(r) = 0$
%is preserved by uniform convergence.
\end{proof}

By choosing $\gamma = n-2$, it follows that $\|\mu\|_\Kato < \infty$.

\begin{lemma} \label{prop:H1embedding}
If $\mu$ is a compactly supported $\alpha$-dimensional measure for some $\alpha > n-2$,
then $\dot{H}^1(\R^n) \subset L^2(\mu)$.
\end{lemma}
\begin{proof}
Two mapping bounds follow directly from the definition of the global Kato norm, using that the integral kernel of $(-\Delta)^{-1}(x,y)$ is a scalar multiple of $|x-y|^{2-n}$, 
\begin{align*}
\|(-\Delta)^{-1}\mu f\|_{L^\infty(\mu)} &\les \|\mu\|_\Kato \|f\|_{L^\infty(\mu)}, \\
\|(-\Delta)^{-1}\mu f\|_{L^1(\mu)} &\les \|\mu\|_\Kato \|f\|_{L^1(\mu)}. 
\end{align*}
Interpolation between these two endpoints yields
\begin{equation}\label{eq:L2Vembed}
\|(-\Delta)^{-1}\mu f\|_{L^2(\mu)} \les \|\mu\|_\Kato \|f\|_{L^2(\mu)}, 
\end{equation}
This, along with a $TT^*$ argument show that the square root $(-\Delta)^{-\frac12}$ is a bounded operator from $L^2(\R^n)$
to $L^2(\mu)$, by duality it is also bounded from $L^2(\mu)^*$ to $L^2(\R^n)$.  At the same time $(-\Delta)^{-\frac12}$ is an isometry from $L^2(\R^n)$ 
onto $\dot{H}^1(\R^n)$.  This suffices to prove the desired inclusion.   Further, \eqref{eq:L2Vembed} shows that 
\begin{equation*}
\|(-\Delta)^{-1}g\|_{L^2(\mu)} \les \|\mu\|_\Kato \|g\|_{L^2(\mu)^*}. 
\end{equation*}
\end{proof}

Given a point $z \in \R^n$, define the translation operator $\tau_z f(x) := f(x-z)$.
Translation operators are not bounded on $L^2(\mu)$ in general,
but they behave quite well when restricted to the subspace $\dot{H}^1$.
Let $j: \dot{H}^1(\R^n) \to L^2(\mu)$ be the natural inclusion operator.

\begin{lemma} \label{lem:HolderKato}
If $\mu$ is a compactly supported $\alpha$-dimensional measure for some $\alpha > n-2$,
\begin{equation}
\|(j - \tau_z)\vphi\|_{L^2(\mu)} \les \sqrt{C_\mu} |z|^{\beta} \|\vphi\|_{\dot{H}^1}
\end{equation}
for any $0< \beta < \frac{\alpha- (n-2)}{2}$ and $|z| < 1$. 
\end{lemma}

\begin{proof}
By a $TT^*$ argument, it suffices to show that
\begin{equation*}
	\|(j-\tau_z)(-\Delta)^{-1}(j^* - \tau_{-z})g\|_{L^2(\mu)} \les C_\mu |z|^{2\beta}\|g\|_{L^2(\mu)^*},
\end{equation*}
where $j^*$ is the inclusion of $L^2(\mu)^*$ into $\dot{H}^{-1}$.  Translations commute
with powers of the Laplacian, so there is another equivalent statement
\begin{equation*}
\|(2 - \tau_z - \tau_{-z})(-\Delta)^{-1}\mu f\|_{L^2(\mu)} \les C_\mu |z|^{2\beta} \|f\|_{L^2(\mu)}.
\end{equation*}
Here we use that $j$ and $j^*$ may be replaced with the operators $\tau_0$ or $\mathbf 1$ on their respective domains.
We now show that
\begin{equation} \label{eq:HolderKato}
\Big| \int_{\R^n} \Big(\frac{2}{|x-y|^{n-2}} - \frac{1}{|x - (y-z)|^{n-2}} - \frac{1}{|x - (y+z)|^{n-2}}\Big) \mu(dx) \Big|
\les C_\mu |z|^{2\beta}.
\end{equation}
Indeed, Lemma~\ref{prop:dimensionKato} immediately proves this bound for the integral
over the region where $|x-y| \leq 2|z|^{2\beta/(\alpha - n + 2)}$.  Since the exponent $2\beta/(\alpha - n + 2)$
is strictly less than 1, the region includes all three singularities at $x \in \{ y, y-z, y+z\}$.

Outside of the region, Taylor's remainder theorem controls the integrand by a multiple of $\frac{|z|^2}{|x-y|^n}$.  We write
$$
	\frac{1}{|x-y|^n}=\frac{1}{|x-y|^{\f{\alpha-n+2}{\beta}+n-\alpha-2}} \frac{1}{|x-y|^{\alpha+2-\f{\alpha-n+2}{\beta}}},
$$
and note that under the hypotheses, both exponents are positive.
On the region of interest, the first term is dominated by $|z|^{2\beta-2}$.  
The estimate continues with
\begin{align*}
\int_{|x-y| > |z|^{\frac{2\beta}{\alpha - n +2}}} \frac{|z|^2}{|x-y|^n}|\mu|(dx) &\les 
\int_{\R^n} \frac{|z|^2}{|z|^{2-2\beta} |x-y|^{\alpha + 2 - \frac{a - n + 2}{\beta}}} |\mu|(dx) \\
&= \int_{\R^n} \frac{|z|^{2\beta}}{|x-y|^{\alpha + 2 - \frac{\alpha - n+2}{\beta}}} |\mu|(dx) \les C_\mu |z|^{2\beta}.
\end{align*}
The last inequality follows from Lemma~\ref{prop:dimensionKato} since $\frac{\alpha - n + 2}{\beta} > 2$, making the
exponent in the denominator less than $\alpha$.

The bound in~\eqref{eq:HolderKato} implies that
\begin{align*}
\|(2-\tau_z - \tau_{-z})(-\Delta)^{-1}\mu f\|_{L^\infty(\mu)} &\les C_\mu|z|^{2\beta}\|f\|_{L^\infty(\mu)}, \\
\|(2-\tau_z - \tau_{-z})(-\Delta)^{-1}\mu f\|_{L^1(\mu)} &\les C_\mu |z|^{2\beta}\|f\|_{L^1(\mu)}.
\end{align*}
The desired $L^2(\mu)$ bound follows by interpolation.
\end{proof}

\begin{lemma} \label{lem:compact}
The embedding $j:\dot{H}^1(\R^n) \to L^2(\mu)$ is compact.
\end{lemma}

\begin{proof}
Let $\eta_r$ be a standard mollifier supported in a ball of radius $0 < r <1$.
Lemma~\ref{lem:HolderKato} implies that
\begin{equation*}
\|\vphi - (\eta_r * \vphi) \|_{L^2(\mu)} \les r^{\beta}\|\vphi\|_{\dot{H}^1}.
\end{equation*}
Furthermore, if $\chi$ is any smooth cutoff that is identically 1 on the support of $\mu$,
then multiplication by $\chi$ has no effect in $L^2(\mu)$.  Thus
\begin{equation*}
\|\vphi - (\eta_r * \vphi)\chi \|_{L^2(\mu)} \les r^{\beta}\|\vphi\|_{\dot{H}^1}.
\end{equation*}
The operators mapping $\vphi$ to $(\eta_r * \vphi)\chi$ is compact on $\dot{H}^1(\R^n)$,
so it is also compact from $\dot{H}^1$ to $L^2(\mu)$.  We have just shown that they
converge to the inclusion map $j$ as $r \to 0$.
\end{proof}

\begin{corollary} \label{cor:compact}
For any fixed $\lambda \geq 0$, the operator $R_0^+(\lambda^2)\mu$ is compact on $L^2(\mu)$.
\end{corollary}
\begin{proof}
Recall that $R_0(0)$ is the same as $(-\Delta)^{-1}$ in dimensions $n \geq 3$.  Then
$R_0(0)\mu$ is the composition $j(-\Delta)^{-1}j^* \mu$, with both inclusions $j$ and $j^*$ being compact.

For $\lambda > 0$, the free resolvents $R_0^+(\lambda^2)$ map weighted $\dot{H}^{-1}(\R^n)$ into
weighted $\dot{H}^1(\R^n)$.  Then, with $\chi$ again a smooth cutoff to the support of $\mu$, $\chi R_0^+(\lambda^2) \chi$ is a bounded map from $\dot{H}^{-1}$
to $\dot{H}^1$ without additional weights due to the compact support of $\mu$.  We may write
\begin{equation*}
R_0^+(\lambda^2)\mu = j (\chi R_0^+(\lambda^2)\chi) j^* \mu
\end{equation*}
with $j$ and $j^*$ once again being compact.
\end{proof}

\begin{rmk} \label{rmk:continuity}
The derivative $\chi \frac{d}{d\lambda} R_0^+(\lambda^2) \chi$ is also a bounded map from $\dot{H}^{-1}$
to $\dot{H}^1$.  The same argument as above shows that the family of operators
$R_0^+(\lambda^2)\mu: L^2(\mu) \to L^2(\mu)$ are differentiable with respect to $\lambda$. 
\end{rmk}

\begin{proof}[Proof of Proposition~\ref{prop:selfadj}]
We can take advantage of the KLMN theorem~\cite[Theorem X.17]{ReSi2}
to produce a unique self-adjoint operator with the correct quadratic form on
$\dot{H}^1(\R^n)$  provided $\mu$ satisfies the form bound
\begin{equation}\label{eq:formbound}
\Big|\int_{\R^n}|\vphi(x)|^2\,d\mu \Big| \le a\| \vphi\|_{\dot{H}^1}^2 
+ b\|\vphi\|_{L^2}^2
\end{equation}
for some $a < 1$. 
For $|z|<1$ and $\vphi \in \dot{H}^1(\R^n)$, by Lemma~\ref{lem:HolderKato} we have
\begin{align*}
\int_{\R^n} |\vphi(x)|^2\,d\mu -& \int_{\R^n} |\vphi(x)|^2\, d(\tau_z \mu)
= \int_{\R^n} \big(|\vphi(x)|^2 - |\vphi(x+z)|^2\big)\,d\mu \\
&= \int_{\R^n} \big(\vphi(x) - \vphi(x+z)\big)\bar{\vphi}(x)\,d\mu  + \int_{\R^n} \vphi(x+z)\big(\bar{\vphi}(x) - \bar{\vphi}(x+z)\big) \,d\mu \\
\le\ &\|(j - \tau_{-z})\vphi\|_{L^2(\mu)}\big(\|\vphi\|_{L^2(\mu)} + \|\tau_{-z}\vphi\|_{L^2(\mu)}\big)
\les\ |z|^\beta \|\vphi\|_{\dot{H}^1}^2.
\end{align*}
It follows that
\begin{equation*}
\Big|\int_{\R^n} |\vphi(x)|^2\,d\mu - \int_{\R^n} |\vphi(x)|^2 (\mu * \eta_r)(x)\,dx \Big|
\les |r|^\beta \|\vphi\|_{\dot{H}^1}^2,
\end{equation*}
for a mollifier $\eta_r$ supported in a ball radius $r$.  At the same time $\mu * \eta_r$ is
a bounded function for each $r > 0$, so there is a second estimate 
\begin{equation*}
\Big|\int_{\R^n} |\vphi(x)|^2 (\mu * \eta_r)(x)\,dx \Big| \leq C_r \|\vphi\|_{L^2(\R^n)}^2.
\end{equation*}
Allowing $r$ to approach zero shows that~\eqref{eq:formbound} holds with any choice of $a > 0$.
\end{proof}

\begin{prop} \label{prop:eigenvalues}
If $\mu$ is a compactly supported $\alpha$-dimensional signed measure for some $\alpha > n-2$, then $-\Delta + \mu$ has finitely many negative eigenvalues and no other spectrum in $(-\infty, 0)$.
\end{prop}

\begin{proof}
Since \eqref{eq:formbound} holds for some $0<a<1$, the operator $-\Delta + \mu$ is bounded from below.  Then the range of the spectral projection $P_{(-\infty,0)}$ is a closed subspace of $L^2$ contained inside the negative-definite subspace of the quadratic form $((-\Delta + \mu)\vphi, \vphi)$.  On this subspace we also have the bound
\begin{equation*}
\|\vphi\|_{\dot{H}^1}^2 \leq ((-\Delta + \mu)\vphi, \vphi) + a\|\vphi\|_{\dot{H}^1}^2 + b\|\vphi\|_{L^2}^2
\leq  a\|\vphi\|_{\dot{H}^1}^2 + b\|\vphi\|_{L^2}^2
\end{equation*}
and it follows that $\|\vphi\|_{\dot{H}^1} \les \|\vphi\|_{L^2}$.

Consider the factorization
\begin{equation*}
-\Delta + \mu = (-\Delta)^{1/2}[I + (-\Delta)^{-1/2}\mu(-\Delta)^{-1/2}](-\Delta)^{1/2}.
\end{equation*}
The central operator is a compact and self-adjoint perturbation of the identity acting on $L^2(\R^n)$, namely $I + (-\Delta)^{-1/2}j^*\mu j(-\Delta)^{-1/2}$.  Its negative-definite subspace is finite dimensional.  As observed above, the range of $P_{(-\infty,0)}$ is contained in $L^2(\R^n) \cap \dot{H}^1(\R^n)$.  The square-root of the Laplacian maps this space into $L^2$ in a one-to-one manner.  Thus the range of $P_{(-\infty,0)}$ is also finite dimensional, with dimension no larger than the negative-definite space of $I + (-\Delta)^{-1/2}\mu(-\Delta)^{-1/2}$.
\end{proof}

\section{Low energy estimates} \label{sec:LowEnergy}

At this point we establish a uniform bound on the low energy perturbed resolvent as an operator on $L^2(\mu)$.  Specifically, we show
\begin{lemma}\label{lem:lowenergybd}
	Let $\mu$ be a real-valued measure on $\R^n$, $n \geq 3$, with dimension $\alpha > n-2$.
	If $(-\Delta + \mu)$ has no eigenvalues at $\lambda \geq 0$, and (if $n = 3,4$) no resonances at $\lambda = 0$,
	then
	\begin{equation}
	\sup_{|\lambda| \leq L} \| R_\mu^+(\lambda^2)\mu \|_{L^2(\mu) \to L^2(\mu)} \leq C_L < \infty.
	\end{equation}
	For any $L>0$, with a fixed constant that depends on $L$.
\end{lemma}

\begin{proof}
The estimation of perturbed resolvents on a finite interval follows a standard 
procedure.  First, we express the perturbed resolvent $R_\mu^+(\lambda^2)\mu$
using the identity
\begin{equation} \label{eq:Res_identity}
R_\mu^+(\lambda^2)\mu = (I + R_0^+(\lambda^2)\mu)^{-1} R_0^+(\lambda^2)\mu.
\end{equation}
The operators $R_0^+(\lambda^2)\mu: L^2(\mu) \to L^2(\mu)$ are continuous with
respect to $\lambda$, so they are uniformly bounded over any finite interval.
If an inverse $(I + R_0^+(\lambda^2)\mu)^{-1}$ exists at each $\lambda \geq 0$, then the inverses
will be continuous, and uniformly bounded on each finite interval.

Suppose $I + R_0^+(\lambda_0^2)\mu$ fails to be invertible on $L^2(\mu)$
for some $\lambda_0 > 0$.
By the Fredholm alternative, there must exist a nontrivial $\psi \in L^2(\mu)$
belonging to its null space.  This function satisfies the bootstrapping relation
\begin{equation*}
\psi = -R_0^+(\lambda_0^2)\mu \psi.
\end{equation*}

Assuming $\mu$ is real-valued, the duality pairing
$ (\mu\psi, \psi) = \int_{\R^n} |\psi^2(x)|\,d\mu$ is real-valued as well.
It is also equal to the pairing 
\begin{equation*}
-(\mu\psi, R_0^+(\lambda_0^2)\mu\psi) = \int_{\R^n} \frac{|\widehat{\mu\psi}(\xi)|^2}{|\xi|^2 - (\lambda_0+i0)^2}\,d\xi
\end{equation*}
whose imaginary part is a multiple of $\int_{\{|\xi| = \lambda_0\}} |\widehat{\mu\psi}(\xi)|^2$.
In order for this quantity to be real, the Fourier transform of $\mu\psi$ must vanish on the sphere of radius $\lambda_0$.

Let $\eta$ be a mollifier whose Fourier transform is identically 1 when $|\xi| \leq 2\lambda_0$.
Convolution against $\eta$ is a bounded operator on $\dot{H}^{-1}$ and it maps finite measures on $\R^n$
to $L^p(\R^n)$, $1 \leq p \leq \infty$.  In particular, $\eta * \mu\psi \in L^{\frac{2n+2}{n+5}}(\R^n)$, along with the fact that its Fourier transform vanishes on the sphere of radius $\lambda_0$,
it follows from~\cite[Theorem~2]{Go16} that $R_0^+(\lambda_0^2)(\eta * \mu\psi) \in L^2(\R^n)$. 

Meanwhile $\mu\psi - (\eta * \mu\psi) \in \dot{H}^{-1}(\R^n)$, and it has Fourier support where $|\xi| > 2\lambda_0$.
On this region the free resolvent multiplier $(|\xi|^2 - \lambda_0^2)^{-1}$ is dominated by $|\xi|^{-1}$, hence we see that $R_0^+(\lambda_0^2)(\mu\psi - (\eta * \mu\psi)) \in L^2(\R^n)$.  This shows that $R_0^+(\lambda_0^2)\mu \psi\in L^2(\R^n)$ and hence $\psi\in L^2(\R^n)$, which contradicts the assumption that $\lambda_0>0$ is not an eigenvalue.  Hence $I + R_0^+(\lambda_0^2)\mu$ is invertible.

\end{proof}

\begin{rmk}\label{rmk:eval}
With the stronger assumption $\alpha > n - \frac{2}{n-1}$, one can follow the argument
in~\cite[Proposition~7]{Go16} to show that
$W^{\frac{1}{n+1}, \frac{2n+2}{n-1}} \subset L^2(\mu)$,
then apply~\cite{KoTa06} to conclude that there are no positive eigenvalues of $(-\Delta + \mu)$.
\end{rmk}

\section{High energy estimates} \label{sec:HighEnergy}

The estimates for $R_0^+(\lambda^2)\mu$ in the preceding sections are adequate for finite intervals
of $\lambda$, however the sharp weighted $L^2(\R^n)$ resolvent bound from~\cite{AgHo76}
only implies that
\begin{equation*}
\|R_0^+(\lambda^2)\|_{\dot{H}^{-1} \to \dot{H}^1} \les 1+ |\lambda|.
\end{equation*}
At high energy one needs to take advantage of the fact that for $f \in L^2(\mu)$, $\mu f$ is not a generic element 
of $\dot{H}^{-1}(\R^n)$.  Our main observation at high energy is that the free resolvent in fact has asymptotic
decay as an operator on $L^2(\mu)$.

\begin{theorem} \label{thm:resolvent}
Suppose $\mu$ is a compactly supported measure of dimension $\alpha > n-(1+\frac{1}{n-1})$.
There exists $\eps > 0$ so that the free resolvent satisfies
\begin{equation} \label{eq:resolvent}
\|R_0^+(\lambda^2)\mu f\|_{L^2(\mu)} \les \la \lambda\ra^{-\eps}
\|f\|_{L^2(\mu)}.
\end{equation}
\end{theorem}

There are close connection between the free resolvent $R_0^+(\lambda^2)$
and the restriction of Fourier transforms to the sphere $\lambda\mathbb{S}^2$.
We make use of a Fourier restriction estimate proved by Du and Zhang~\cite{DuZh19}.
Theorem 2.3 of~\cite{DuZh19} asserts that for a function $f \in L^2(\R^{n-1})$ with 
Fourier support in the unit ball, and a measure $\mu_R = R^{\alpha} \mu(\,\cdot\,/R)$,
\begin{equation*}
\|e^{it\Delta}f\|_{L^2(\mu_R)} \les R^{\frac{\alpha}{2n}} \|f\|_{L^2(\R^{n-1})},
\end{equation*}
for sufficiently large $R$. The Schr\"odinger evolution $e^{it\Delta}f$ is the inverse Fourier transform (in $\R^n$) of the
measure $\hat{f}\in L^2(\R^{n-1})$ lifted onto the paraboloid $\Sigma = \{\xi_n = |\xi_1|^2 + \cdots + |\xi|_{n-1}^2\}$. 
The theorem is then equivalent to the statement
\begin{equation*}
\| (\widehat{g d\Sigma})\|_{L^2(\mu_R)} \les R^{\frac{\alpha}{2n}} \|g\|_{L^2(\Sigma)}
\end{equation*}
for functions $g \in L^2(\Sigma \cap B(0,1))$.  The use of forward versus inverse Fourier transform does not affect the inequality.

It is well known that the bounded subset of the paraboloid $\Sigma$ can be
replaced with any other uniformly convex bounded smooth surface.  In this case we wish to apply the result to the unit
sphere instead.  For any $g \in L^2(S^{n-1})$,
\begin{equation*}
\| \hat{g} \|_{L^2(\mu_R)} \les R^{\frac{\alpha}{2n}} \| g\|_{L^2(S^{n-1})}.
\end{equation*}
The dual statement is
\begin{equation*}
\big \| \widehat{\mu_Rf}\big|_{|\xi| = 1} \big\|_{L^2(S^{n-1})} \les R^{\frac{\alpha}{2n}} \|f\|_{L^2(\mu_R)}.
\end{equation*}

Now we reverse some of the scaling relations.  Given $f \in L^2(\mu)$, let $f_R(x) = R^{-\frac{\alpha}{2}}f(x/R)$
so that $\|f_R\|_{L^2(\mu_R)} = \|f\|_{L^2(\mu)}$.  Then $\widehat{\mu_Rf_R}(\xi) = R^{\frac{\alpha}{2}} \widehat{\mu f}(R\xi)$. 
It follows that
\begin{equation} \label{eq:UsefulRestriction}
\big\| \widehat{\mu f} \big|_{|\xi| = R} \big\|_{L^2(R S^{n-1})}
= R^{\frac{n-1-\alpha}{2}}\big\|\widehat{\mu_Rf_R} \big|_{|\xi| = 1} \big\|_{L^2(S^{n-1})}
\les R^{\frac{n-1}{2} - \frac{\alpha}{2}(1 - \frac{1}{n})} \|f\|_{L^2(\mu)}.
\end{equation}

Thanks to the compact support of $\mu$, the $L^2(\mu)$ norm of $(1+ |x|)f$ is comparable to that of $f$.
That allows for control of the derivatives of $\widehat{\mu f}$ with the same restriction bound as in~\eqref{eq:UsefulRestriction}.
In particular we can bound the outward normal gradient of $\widehat{\mu f}$ as
\begin{equation} \label{eq:UsefulRestriction2}
\Big\|{\textstyle \frac{\xi}{|\xi|}}
 \cdot\nabla_\xi(\widehat{\mu f})(\xi)\big|_{|\xi|=R}\Big\|_{L^2(RS^{n-1})}
\les R^{\frac{n-1}{2} - \frac{\alpha}{2}(1-\frac{1}{n})} \|f\|_{L^2(\mu)}.
\end{equation}

\begin{proof}[Proof of Theorem~\ref{thm:resolvent}]
In light of Lemma~\ref{lem:lowenergybd}, we need only consider $|\lambda|\gtrsim 1$.
The specific inequality we derive has the form
\begin{equation} \label{eq:Resolvent_High}
\|R_0^+(\lambda^2)\mu f\|_{L^2(\mu)} \les \lambda^{n-2 - \alpha(\frac{n-1}{n})}\log \lambda
\|f\|_{L^2(\mu)}
\end{equation} Our assumption $\alpha > n - (1 + \frac{1}{n-1})$ is chosen to make
the exponent negative on the right-hand side.

The free resolvent $R^\pm_0(\lambda^2)$ acts by multiplying Fourier transforms
pointwise by the distribution
\begin{equation*}
\frac{1}{|\xi|^2 - \lambda^2} 
\pm i\frac{\pi}{\lambda}\,d\sigma(|\xi| = |\lambda|).
\end{equation*}
The surface measure term is $\frac{C}{\lambda}$ times the $T^{*}T$ composition of
the operator in~\eqref{eq:UsefulRestriction}.  Thus it has an operator norm bound
controlled by $\lambda^{[n-1 - \alpha(\frac{n-1}{n})] - 1}$.

The multiplier term can be split into two parts depending on whether $|\xi|\sim \lambda$ or not.
Let $\phi$ be a smooth function supported in the annulus $\frac12 \leq |\xi|< 2$ that is identically 1 when
$\frac34 \leq |\xi| \leq \frac32$.  Note that there is an upper bound
\begin{equation*}
\frac{1 - \phi(\xi/\lambda)}{|\xi|^2 - \lambda^2} \les \min(\lambda^{-2}, |\xi|^{-2}) 
\leq \frac{1}{\lambda^{\gamma - (n-2)}|\xi|^{n-\gamma}}
\end{equation*} 
for any $n-2 \leq \gamma \leq n$.  Hence this part of the free resolvent maps $\dot{H}^{-\frac{n-\gamma}{2}}(\R^n)$
to $\dot{H}^{\frac{n-\gamma}{2}}(\R^n)$ with an operator norm less than $\lambda^{-(\gamma - (n-2))}$.

The proof of Lemma~\ref{prop:H1embedding} can be modified trivially to show that
there is a continuous embedding $j: H^{\frac{n-\gamma}{2}}(\R^n) \mapsto L^2(\mu)$ whenever $\gamma < \alpha$
(as required by Lemma~\ref{prop:dimensionKato}).  Since $\alpha > n-2$ this includes
a nonempty interval $\gamma \in [n-2, \alpha)$.

By expanding out the free resolvent as $R_0^+(\lambda^2)\mu = j R_0^+(\lambda^2) j^* \mu$, we see that
frequencies $|\xi| \not\sim \lambda$ give rise to an operator on $L^2(\mu)$ with norm bound
$\lambda^{-(\gamma - (n-2))}$.

The portion of the free resolvent with frequency $|\xi| \sim \lambda$ will be handled by restricting
$\widehat{\mu f}$ to spheres of radius $s$, then integrating the results.
For each $\frac{\lambda}{2}< s <2\lambda$ define
\begin{equation*}
F_s(x) := \mu f * \widecheck{d\sigma}(\,\cdot\,/s)
\end{equation*}
with $d\sigma$ being the surface measure of the unit sphere.
This way, $s^{n-1}\widehat{F}_s$ is the restriction of $\widehat{\mu f}$ to the sphere $|\xi| = s$.
By the Parseval identity we have an inner product formula for $f, g \in L^2(\mu)$,
\begin{equation*}
\langle F_s, g\rangle_{L^2(\mu)} = \int_{S^{n-1}} \widehat{\mu f}(s\omega) \overline{\widehat{\mu g}(s\omega)}\,d\omega
= s^{1-n} \la \widehat{\mu f}, \widehat{\mu g}\rangle_{L^2(sS^{n-1})}.
\end{equation*}

Inequality~\eqref{eq:UsefulRestriction} shows that $\|F_s\|_{L^2(\mu)} \les s^{-\alpha(\frac{n-1}{n})}\|f\|_{L^2(\mu)}$.
If one takes a derivative with respect to $s$, it is easy to apply the product rule to the middle expression.
Then the bounds~\eqref{eq:UsefulRestriction} and~\eqref{eq:UsefulRestriction2} show that 
$\|\frac{d}{ds}F_s\|_{L^2(\mu)} \les s^{-\alpha(\frac{n-1}{n})}\|f\|_{L^2(\mu)}$ as well.

Now the remaining part of the free resolvent appears as a principal value integral
\begin{equation} \label{eq:pvIntegral}
\Big\|p.v. \int_{\lambda/2}^{2\lambda} 
\Big(\frac{s^{n-1} \phi(\frac{s}{\lambda})}{s+\lambda} F_s\Big) 
 \frac{1}{s-\lambda}\,ds\Big\|_{L^2(\mu)}.
\end{equation}

Based on the discussion above, both $s^{n-1}\phi(\frac{s}{\lambda})F_s/(s+\lambda)$ and
$\frac{d}{ds}\big[s^{n-1}\phi(\frac{s}{\lambda})F_s/(s+\lambda)\big]$ are bounded in $L^2(\mu)$
with norm less than $\lambda^{n-2-\alpha(\frac{n-1}{n})}$ so long as $s \sim \lambda$ and $\lambda \geq 4$.
The desired bound~\eqref{eq:Resolvent_High} follows by bringing these norms inside the integral
when $|s-\lambda| > 1$, and integrating by parts once in the more singular interval $|s-\lambda| \leq 1$
before bringing the norms inside.  The resulting integral in each case is bounded by $\log \lambda$.
\end{proof}

We are now able to prove the uniform resolvent bounds in Theorem~\ref{thm:resolvents} and consequently the Strichartz estimates in Theorem~\ref{thm:main}.

\begin{proof}[Proof of Theorem~\ref{thm:resolvents}]
The uniform bound~\eqref{eq:UnifFreeRes} combines low-energy existence from Corollary~\ref{cor:compact}, uniformly on bounded intevals of $\lambda$ from Remark~\ref{rmk:continuity}, and  decay as $\lambda \to \infty$ from Theorem~\ref{thm:resolvent}.  The low-energy part of~\eqref{eq:UnifPerturbedRes} is stated as Lemma~\ref{lem:lowenergybd}.
At high energies, we once again apply the resolvent identity~\eqref{eq:Res_identity}.  Theorem~\ref{thm:resolvent} provides decay of $R_0^+(\lambda^2)\mu$, and once its norm is less than $\frac12$, then the perturbation $(I + R_0^+(\lambda^2)\mu)^{-1}$ and consequently $R_\mu^\pm(\lambda^2)$ are uniformly bounded as well.
\end{proof}

\begin{proof}[Proof of Theorem~\ref{thm:main}]
The derivation of local decay estimates~\eqref{eq:freesmoothing} and~\eqref{eq:smoothing} for the Schr\"odinger equation from uniform resolvent bounds follows Kato's argument~\cite{Ka65} with minimal adaptation. One can freely interchange the order of the $L^2_t$ and $L^2(\mu)$ norms.  Then by a $TT^*$ argument, and using the fact that multiplication by $\mu$ is a unitary map between $L^2(\mu)$ and its dual space, 
\begin{equation*}
\|e^{it\Delta}f  \|_{L^2(\mu)L^2_t} \leq C\|f\|_2 \quad \text{if and only if} \quad
\Big\|\int_\R e^{i(t-s)\Delta}\mu g(\,\cdot\,, s)\,ds \Big\|_{L^2(\mu)L^2_t} \leq C^2 \|g\|_{L^2(\mu)L^2_t}.
\end{equation*}
After applying Plancherel's identity to a partial Fourier transform in the time variable, with $\lambda$ as the dual variable to $t$, this is again equivalent (up to a constant) to the bound
\begin{equation*}
\sup_{\lambda \geq 0} \|(R_0^+(\lambda) - R_0^-(\lambda))\mu\|_{L^2(\mu)\to L^2(\mu)} \leq C^2.
\end{equation*}

The derivation of~\eqref{eq:smoothing} is identical except that the Fourier transform of $e^{i(t-s)(-\Delta+\mu)}P_{ac}$ generates the difference of perturbed resolvents $R_\mu^+(\lambda) - R_\mu^-(\lambda)$.  Negative values of $\lambda$ are excluded because the absolutely continuous spectrum of $-\Delta + \mu$ is still $[0,\infty)$.

The Strichartz inequalities are proved via the argument by Rodnianski and Schlag~\cite{RoSc04}. Use Duhamel's formula to write out the perturbed evolution as
\begin{equation*}
e^{it(-\Delta+\mu)}P_{ac}f = e^{-it\Delta}P_{ac}f + i\int_0^te^{-i(t-s)\Delta}\mu e^{is(-\Delta+\mu)}P_{ac}f\,ds
\end{equation*}
for $t>0$.  Note that $P_{ac}$ is an orthogonal projection, so it is a bounded operator on $L^2(\R^n)$. 
%By assumption, there are at most finitely many eigenvalues of $-\Delta+\mu$, all with exponentially localized eigenfunctions, so $P_{ac}$ is a bounded operator on $L^2(\R^n)$.
The free evolution term satisfies all Strichartz inequalities including the $p=2$ endpoint.  For the inhomogeneous term, our local decay bound~\eqref{eq:smoothing} shows that
$\mu e^{is(-\Delta+\mu)}P_{ac}f \in L^2_t L^2(\mu)^*$.  The dual statement to~\eqref{eq:freesmoothing} together with the free Strichartz inequalities imply that
\begin{equation*}
\Big\| \int_\R e^{-i(t-s)\Delta} G(\,\cdot\,, s)\,ds \Big\|_{L^p_t L^q_x} \les \|G\|_{L^2_t L^2(\mu)^*}.
\end{equation*}
An application of the Christ-Kiselev lemma (for example as stated in \cite{RoSc04}, Lemma 4.2) shows that the same bound holds for the desired domain of integration $0 \leq s \leq t$ provided $p>2$.
\end{proof}

\end{document}